\documentclass[12pt]{amsart}
\usepackage{mathrsfs}
\usepackage{eucal}
\usepackage{color}
\usepackage[all]{xy}
\usepackage{amssymb,amsmath}

\date{Sept.~14, 2010}
\calclayout
\newtheorem{dummy}{anything}[section]
\newtheorem{theorem}[dummy]{Theorem}
\newtheorem*{thma}{Theorem A}
\newtheorem*{thmb}{Theorem B}

\newtheorem{lemma}[dummy]{Lemma}

\newtheorem{corollary}[dummy]{Corollary}

\theoremstyle{definition}

  \newtheorem{example}[dummy]{Example}
  
  \newtheorem{remark}[dummy]{Remark}

  \newtheorem*{acknowledgement}{Acknowledgement}

\newcommand
{\eqncount}{\setcounter{equation}{\value{dummy}}%
\addtocounter{dummy}{1}}

\newcommand{\cI}{\mathscr I}

\newcommand{\cL}{\mathcal L}

\newcommand{\cN}{\mathcal N}
\newcommand{\cS}{\mathscr S}

\newcommand{\bZ}{\mathbb Z}

\newcommand{\bP}{\mathbf P}
\newcommand{\CP}{\mathbf C \bP}

\newcommand{\bbC}{\mathbb C}

\newcommand{\bbL}{\mathbb L}

\newcommand{\bbQ}{\mathbb Q}
\newcommand{\bbR}{\mathbb R}

\newcommand{\bbZ}{\mathbb Z}

\DeclareMathOperator{\Hom}{Hom}
\DeclareMathOperator{\Image}{im}
\DeclareMathOperator{\hdim}{hom\,dim}
\DeclareMathOperator{\invlim}{\lower 6pt\hbox{$\stackrel{\displaystyle{\lim}}{\leftarrow}$}}
\DeclareMathOperator{\dirlim}{\lower 6pt\hbox{$\stackrel{\displaystyle{\lim}}{\rightarrow}$}}

\newcommand{\wO}{\widetilde{\Omega}\hphantom{}}
\newcommand{\wKO}{\widetilde{KO}\hphantom{}}

\newcommand{\cy}[1]{\bZ/{#1}}
\newcommand{\la}{\langle}
\newcommand{\ra}{\rangle}

\newcommand{\bd}{\partial}
\newcommand{\id}{\mathrm{id}}

\def\Zp{\bZ_{(p)}}

\def\Zodd{\bZ[1/2]}

\def\Ztwo{\bZ_{(2)}}

\def\La{\Lambda}

\DeclareMathOperator{\pt}{pt}
\DeclareMathOperator{\eval}{\mathbf{eval}}

\newcommand{\Zpi}{\bZ\pi}

\newcommand{\bast}{\ast}
\newcommand{\SC}{\mathfrak C}
\newcommand{\SI}{\mathfrak I}
\newcommand{\SA}{\mathfrak A}
\newcommand{\ko}{\mathbf{ko}}
\newcommand{\ku}{\mathbf{ku}}
\newcommand{\bo}{\mathbf{bo}}
\newcommand{\bu}{\mathbf{bu}}
\newcommand{\Kpi}{B\pi}
\DeclareMathOperator{\sign}{index}
\newcommand{\co}{[1]}

\begin{document}

\title{Surgery obstructions on closed manifolds and the Inertia subgroup}
\author{Ian Hambleton}

\address{Department of Mathematics, McMaster University,  Hamilton, Ontario L8S 4K1, Canada}

\email{hambleton@mcmaster.ca }

\thanks{Research partially supported by NSERC Discovery Grant A4000. The author would like  to thank the Max Planck Institut f\"ur Mathematik in Bonn for its hospitality and support while this work was done.}

\begin{abstract}
The Wall surgery obstruction groups have two interesting geometrically defined subgroups, consisting of the surgery obstructions between closed manifolds, and the inertial elements. 
We show that  the inertia group $I_{n+1}(\pi,w)$ and the closed manifold subgroup $C_{n+1}(\pi,w)$ are equal in dimensions $n+1\geq 6$, for any finitely-presented group $\pi$ and any orientation character $w\colon \pi \to \cy 2$. This answers a question from \cite[p.~107]{h3}. 
\end{abstract}

\maketitle

\section{Introduction}
\label{sect: introduction}
Let $\pi$ be a finitely-presented group, and let $L_n(\Zpi,w)$ denote Wall's surgery obstruction group for oriented surgery problems up to simple homotopy equivalence, where $w\colon \pi \to\cy 2$ is an orientation character (see \cite[Chap.~5-6]{wall-book}). We work with topological (rather than smooth) manifolds throughout, so rely on the work of Kirby-Siebenmann \cite{kirby-siebenmann1} for the extension of surgery theory to the topological category. 

\smallskip
Let $X^{n}$ be a closed, topological $n$-manifold, $n \geq 5$, and let $c\colon X \to \Kpi$ denote the classifying map of its universal covering, so that $c_*\colon \pi_1(X, x_0) \xrightarrow{\approx} \pi$ is a given isomorphism. The orientation class
$w_1(X) \in H^1(X;\cy 2)$  induces an orientation character $w\colon \pi \to \cy 2$.
 The surgery exact sequence
 $$  [\Sigma(X), G/TOP] \xrightarrow{\sigma_{n+1}(X)} L_{n+1}(\Zpi,w) \longrightarrow \cS(X) \longrightarrow [X, G/TOP] \xrightarrow{\sigma_n(X)} L_n(\Zpi,w)$$
 developed by Browder, Novikov, Sullivan and Wall \cite[Chap.~9]{wall-book} relates the classification of manifolds which are simple homotopy equivalent to $X$ to the calculation of the surgery obstruction maps $\sigma_{n+1}(X)$ and $\sigma_n(X)$. 
 
In the surgery exact sequence $\Sigma(X) = (X\times I)/\bd(X\times I)$, and $\cS(X)$ denotes the $s$-cobordism classes of pairs $(M,f)$, where $f\colon M \to X$ is a simple homotopy equivalence. For a suitable $H$-space structure on $G/TOP$ (see \cite{kirby-siebenmann1}, \cite{madsen-milgram1}, \cite{ranicki-assembly}), these surgery obstruction maps are  homomorphisms between abelian groups.

 For a fixed $(X,w)$, let $C_{n}(X,w) \subseteq L_{n}(\Zpi,w)$
  denote the image of $\sigma_{n}(X)$.  This is the subgroup of $L_{n}(\Zpi,w)$ given by the surgery obstructions of all degree 1 normal maps $(f,b)\colon M \to X$ from some closed $n$-manifold $M$. By varying $X$ over all closed manifolds with the same orientation character $w$, we define the \emph{closed manifold subgroup}
  \eqncount
\begin{equation}\label{closed_subgp}
 C_{n}(\pi,w) \subseteq L_{n}(\Zpi,w)\end{equation}
  as the subgroup of the $L$-group generated by all of the closed manifold  subgroups $C_{n}(X,w)$.
  In the oriented case ($w\equiv 1$), $C_n(\pi)$ is just the image of the 
Sullivan-Wall homomorphism \cite[13B.3]{wall-book}
$$\Omega_n(\Kpi \times G/TOP, \Kpi \times \ast) \to L_n(\Zpi)$$
defined by the surgery obstruction. 

 For a fixed $(X,w)$, let $I_{n+1}(X,w) \subseteq L_{n+1}(\Zpi,w)$
  denote the image of $\sigma_{n+1}(X)$.  This is the subgroup of $L_{n+1}(\Zpi,w)$ which acts trivially on the structure set $\cS(X)$. The surgery exact sequence (and the $s$-cobordism theorem) shows that the elements of $I_{n+1}(X,w)$ are exactly the surgery obstructions of \emph{relative} degree 1 normal maps 
  $$(f,b)\colon (W, \bd W) \to (X\times I, X \times \bd I),$$ where the $f$ restricted to the boundary $\bd W$ is a homeomorphism. By glueing a copy of $X \times I$ along the boundary components in domain and range, we obtain a closed manifold surgery problem $W \cup_{\bd W} (X \times I) \to X \times S^1$. By varying $X$ over all closed manifolds with the same orientation character $w$, we define the \emph{inertia subgroup}
  \eqncount
  \begin{equation}\label{inertia_subgp}
   I_{n+1}(\pi,w) \subseteq L_{n+1}(\Zpi,w)\end{equation}
  as the subgroup of the $L$-group generated by all of the inertia groups $I_{n+1}(X,w)$. By construction, $I_{n+1}(\pi,w) \subseteq C_{n+1}(\pi,w)$ for all fundamental group data $(\pi,w)$, and $n \geq 5$.
  Here is our main result:
\begin{thma} Let $\pi$ be a finitely-presented group and  $w\colon \pi \to \cy 2$ an orientation character. The inertia subgroup $I_{n+1}(\pi,w)$ equals the group of closed manifold surgery obstructions $C_{n+1}(\pi,w) \subseteq L_{n+1}(\Zpi),w)$, for all $n\geq 5$.
\end{thma}
As stated, this holds for the simple surgery  obstructions in $L^s_{n+1}(\Zpi,w)$, but the inertial or closed manifold subgroups of $L^h_{n+1}(\Zpi,w)$ are just the images of $I_{n+1}(\pi,w)$ or $C_{n+1}(\pi,w)$ under the natural change of $K$-theory homomorphism $L^s \to L^h$ (or $L^h \to L^p$). It follows that the inertial and closed manifold subgroups are equal for all torsion decorations in $K_i(\Zpi)$, $i \leq 1$.
In \cite{h3} it was proved that the images of these two subgroups were equal in the \emph{projective} surgery obstruction groups $L^p_{n+1}(\Zpi,w)$, for $\pi$ a finite group, and the question answered here was raised in \cite[p.~107]{h3}. 
\begin{remark}
Fairly complete information is available about the closed manifold obstructions for finite fundamental groups \cite[Theorem~A]{hmtw1}, under the assumptions that the manifolds are oriented and surgery obstructions are measured up to \emph{weakly simple} homotopy equivalence, with Whitehead torsion in $SK_1(\Zpi)$. The outstanding open problems in this area are (i) to investigate the non-oriented case, (ii) to compute the \emph{simple} closed manifold obstructions in $L_*^s$, and (iii) to decide whether the component 
$\kappa_4\colon H_4(\pi;\cy 2) \to L_6(\Zpi)$ of the assembly map $A_*\colon H_*(\Kpi^w;\bbL_\bullet) \to L_*(\Zpi,w)$ is zero or non-zero (see Section \ref{sec:one} and \cite[p.~352]{hmtw1} for this notation). 
\end{remark}

For  a finitely-presented group $\pi$, the closed manifold subgroup $C_n(\pi,w)$  is contained in the image $A_n(\pi,w)$ of the assembly map, but for $\pi$ of infinite order they are not always equal (see Example \ref{ko-example}).  However,  these subgroups do become equal after localizing at $2$ (see Theorem \ref{thm:2-local}), or after stabilizing as follows. The periodicity isomorphism
$$L_{n}(\Zpi,w) \xrightarrow{\times \CP^2} L_{n+4}(\Zpi,w)$$
allows us to identify $L_n \cong L_{n+4k}$, for all $k \geq 0$. We define the \emph{periodic} image of the assembly map  $\SA_{q}(\pi,w)$, $0 \leq q \leq 3$, as the subgroup of $L_q(\Zpi,w)$ generated by all of the images of the assembly maps  $A_{n}(\pi,w)$, for $n \equiv q \pmod 4$. 

Similarly, we define the \emph{periodic} inertial subgroup $\SI_{q}(\pi,w)$ and 
the  \emph{periodic} closed manifold subgroup $\SC_{q}(\pi,w)$, $0 \leq q \leq 3$, as the subgroups of $L_q(\Zpi,w)$ generated by all  $I_n(\pi,w)$ and $C_n(\pi,w)$, respectively, for $n \equiv q \pmod 4$.  
After stabilization we obtain a result for all fundamental groups.
\begin{thmb}
Let $\pi$ be a finitely-presented group and $w\colon \pi \to \cy 2$  an orientation character.  The periodic inertial subgroup $\SI_q(\pi,w)$ and the  periodic closed manifold subgroup $\SC_{q}(\pi,w)$ both equal the  periodic image of the assembly map $\SA_q(\pi,w)  \subseteq L_{q}(\Zpi),w)$, for $0\leq q\leq 3$.
\end{thmb}

\begin{remark}
For infinite torsion-free groups, the $L$-theory assembly maps are conjectured to be isomorphisms \cite{farrell-jones1}, and this is currently an active area of research. For infinite groups $\pi$ with torsion, conjecturally the contribution to $C_n(\pi)$ arising from finite subgroups is determined by the virtually cyclic subgroups of $\pi$. 
Theorem B is proved by showing that  the periodic inertial subgroup $\SI_{q}(\pi,w)$ is equal to a periodic stabilization of the image of the assembly map (see Theorem \ref{thm:periodic}).
\end{remark}
\begin{acknowledgement}
The author would like to thank Matthias Kreck,  Larry Smith, Andrew Ranicki, and Larry Taylor for helpful conversations and correspondence. For this revision, I am also indebted to
Bob Bruner, John Greenlees, David Johnson, Peter Landweber, and Steve Wilson for a lot of useful  information about complex bordism and connective $K$-theory.

\end{acknowledgement}
\section{The surgery assembly map}\label{sec:one}
 We will need to use the relationship between the closed manifold subgroup and the image of the $L$-theory assembly map.
  Recall that there is a factorization due to Quinn \cite{quinn0} and Ranicki \cite{ra7}, \cite[\S 18]{ranicki-assembly} (see also Nicas \cite[\S 3]{nicas-memoir}):
 $$\xymatrix@C-15pt{ [X\times I, X \times \bd I; G/TOP, \ast] \ar[rrr]^(0.6){\sigma_{n+1}(X)}\ar[d]^{\cap [X, \bd X]_{\bbL^0}}&&&L_{n+1}(\Zpi,w)\\
 H_{n+1}(X^w; \bbL_\bullet) \ar[r]^{i_\bullet}&H_{n+1}(X^w; \bbL_0)\ar[rr]^{c_*}&&H_{n+1}(\Kpi^w; \bbL_0)\ar[u]_{\ A_{n+1}}
 }$$
of the surgery obstruction map $\sigma_{n+1}(X)$ through the assembly map $A_{n+1}$, where $\bbL_0$ denotes the $(-1)$-connective quadratic $L$-spectrum with  $\bbZ \times G/TOP$ in dimension zero, and $[X, \bd X]_{\bbL^0}$ is the fundamental class for  symmetric $\bbL^0$-theory. Cap product with this fundamental class gives a  Poincar\'e duality isomorphism \cite[18.3]{ranicki-assembly} for  $\bbL_0$, and for its $0$-connective cover $\bbL_{\bullet}$ (which has $G/TOP$ in dimension zero). In particular, 
$$H^0(X, \bd X; \bbL_{\bullet}) = [X\times I, X \times \bd I; G/TOP, \ast] \ \cong \ H_n(X^w;\bbL_{\bullet})\ .$$
The (co)fibration $i_\bullet\colon \bbL_{\bullet} \to \bbL_0 \to H\bZ$ of spectra induces a long exact sequence relating the two homology theories.
Similarly, we have the formula
$$\sigma_n(X) = A_n \circ c_*\circ i_\bullet\circ (-\, \cap\, [X]_{\bbL^0}).$$
 The notation $X^w$ or
$\Kpi^w$ means the Thom spectrum  of the line bundle over $X$ or $\Kpi$ induced by $w$, with Thom class in dimension zero, and the assembly maps $A_*$ are induced by a spectrum-level composite
$$A_{\pi,w}\colon \Kpi^w \wedge \bbL_0 \xrightarrow{a_{\pi,w}\wedge 1}
\bbL^0(\Zpi, w) \wedge \bbL_0 \longrightarrow \bbL_0(\Zpi,w)$$
as described in \cite[\S 1]{hmtw1}. 
In particular, the homomorphisms $A_n$, $n \geq 0$, are just the maps induced on homotopy groups by $A_{\pi,w}$. 
We define the subgroup
\eqncount
\begin{equation}\label{assembly_subgp}
A_{n}(\pi,w) = \Image\left (H_{n}(\Kpi^w; \bbL_\bullet) \xrightarrow{i_\bullet} H_{n}(\Kpi^w; \bbL_0) \xrightarrow{A_n} L_{n}(\Zpi,w)\right ) \end{equation}
as the image of the assembly map restricted to $\bbL_\bullet$, for any dimension $n\geq 0$.  Let $\SA_q(\pi,w)$ denote the image of the assembly map made \emph{periodic}. We observe that the factorization of the surgery obstruction map implies that
$$C_{n}(\pi,w) \subseteq A_{n}(\pi,w) \subseteq L_{n}(\Zpi,w),$$
and
$$I_{n+1}(\pi,w) \subseteq C_{n+1}(\pi,w)\subseteq A_{n+1}(\pi,w),$$
but the inertial subgroup and the closed manifold subgroup have purely geometric definitions independent of the assembly map.
\begin{remark}
In \cite[18.6(i)]{ranicki-assembly} it is stated without proof that $C_n(\pi,w) = A_n(\pi,w)$, for $n\geq 5$. This is not true in general (see Example \ref{ko-example}), but we will verify this for $\pi$ finite. For $\pi$ any finitely-presented group, we show that $C_n(\pi,w)\otimes \Ztwo = A_n(\pi,w)\otimes \Ztwo$, for $n\geq 5$, and that $\SC_q(\pi,w) = \SA_q(\pi,w)$, for $0\leq q \leq 3$. 
\end{remark}

\section{The characteristic class formulas}\label{sec:two}
We will use the characteristic class formulas for the surgery obstruction maps
$\sigma_*(X)$ as presented by Taylor and Williams \cite{taylor-williams1} (see also \cite{wall-1976}, and \cite[\S 1]{hmtw1} for the non-oriented case).  Let $\bo(\La)$ denote the connective $KO$-spectrum with coefficients in a group $\La$. The associated homology theory is called $\ko$-homology (with coefficients in $\La$). The Morgan-Sullivan characteristic class  \cite{morgan-sullivan1} is denoted 
$\cL \in H^{4*}(BSTOP; \Ztwo)$, and $V \in H^{2i}(BSTOP;\cy 2)$ denotes the total Wu class.

\begin{enumerate}\addtolength{\itemsep}{0.2\baselineskip}
\item (\cite[Theorem~A]{taylor-williams1}) The spectra $\bbL^0(\Zpi, w)$ and $\bbL_0(\Zpi,w)$ are generalized Eilenberg-MacLane spectra when localized at 2, and when localized away from 2 are both
$$\bo(\La_0)\vee \bo(\La_1)\vee \bo(\La_2)\vee \bo(\La_3),$$
 where  $\La_i = \pi_i(\bbL^0(\Zpi,w))\otimes \Zodd$. In particular,  there is an equivalence of spectra
$ \bbL^0\otimes \Zodd \simeq \bo(\Zodd)$,  defining a characteristic class
$$\Delta \in KO^0(\bbL_\bullet\,; \Zodd)$$ whose associated map
  $\sigma\colon G/TOP [1/2] \xrightarrow{\sim} BO[1/2]$ is the homotopy equivalence 
of infinite loop spaces due to Sullivan and Kirby-Siebenmann (see \cite{sullivan1}, \cite{madsen-snaith-tornehave1}, and the exposition in \cite[4.28]{madsen-milgram1}).
\item  The splitting of $\bbL_0 \otimes \Ztwo$ is given by universal cohomology classes
$\ell \in H^{4\bast}(\bbL_0\,; \Ztwo)$ and $k\in H^{4\bast+2}(\bbL_0\,;\cy 2)$. The domain of the assembly map 
$$H_n(\Kpi^w; \bbL_0)\otimes \Ztwo \xrightarrow{\approx}
\bigoplus_{i \geq 0} H_{n-4i}(\pi; \Ztwo^w) \oplus H_{n-4i - 2}(\pi; \cy 2)$$
has a natural splitting induced by $\ell$ and $k$. The assembly map has component maps 
$$\cI_m \colon H_{m}(\pi; \Ztwo^w) \to L_m(\Zpi,w)\otimes \Ztwo, \ m \geq 0$$
and
$$\kappa_m\colon  H_{m}(\pi; \cy 2) \to L_{m+2}(\Zpi,w)\otimes \Ztwo, \ m \geq 0$$
which determine $A_{n}(\pi,w)\otimes \Ztwo$ completely (see \cite[\S 1]{hmtw1}).

\item (\cite[Theorem~C]{taylor-williams1}) Let $X$ be a closed $n$-manifold, with a reference map $c\colon X \to \Kpi$ such that $c^*(w) = w_1(X)$, and let $\eta_X$  denote the orientation line bundle  over $X$. 
Let $u_X\colon X \to BSTOP$ classify the bundle $\nu_{+}$ such that $\nu_{+}\oplus \eta_X$  is the stable normal bundle $\nu_X$.  Let $f\colon X \to \bbL_\bullet$ determine a degree 1 normal map. Then 
$$\sigma_X(f)_{(odd)} = A_*c_*\left (f^*(\Delta) \cap [X]_{\bo} \right)$$
gives the surgery obstruction localized away from 2, where $[X]_{\bo} \in \ko_n(X^w)$ denotes the $\ko$-fundamental class of $X$.
Furthermore, 
 the $2$-local surgery obstruction is given by
$$\qquad \ \sigma_X(f)_{(2)} = A_*c_*\left ( (u_X^*(\cL) \cup f^*(\ell) + u_X^*(\cL) \cup f^*(k) + \delta^*(u_X^*(VSq^1V) \cup f^*(k))) \cap [X] \right )$$
where $\delta^*$ denotes the integral Bockstein and $A_*$ is the assembly map.
\end{enumerate}
These formulas translate the given information about the manifold $X$ and the surgery problem $f\colon X \to \bbL_\bullet$ into a collection of $\ko$-homology classes (away from 2), or a collection of ordinary $\Ztwo^w$ or $\cy 2$ homology classes for $(\pi,w)$. The surgery obstruction $\sigma_X(f)\in L_n(\Zpi,w)$ is then computed by applying the assembly map to these classes. There are similar formulas for the obstruction to  a relative surgery problem defined by
$f\colon \Sigma(X) \to \bbL_\bullet$, involving the relative fundamental class $[X\times I, X\times \bd I]$.
\begin{remark} Note that the degree 0 component of the class $f^*(\ell)$ in the $2$-local formula is zero. The class $f^*(\Delta)$ has a similar property which will be made precise in Lemma \ref{lem:positive}.
\end{remark}

\section{The proof of Theorem A (localized at $2$)}\label{sec:three}
We fix the fundamental group data $(\pi,w)$. The idea of the proof  (generalizing \cite[\S 4]{h3}) is to construct enough inertial surgery problems to realize all possible elements of $C_{n+1}(\pi,w)$. The target manifolds for these surgery problems will have the form
$X^n \times I$, where $X^n$ is the total space of an $S^{n-m}$-bundle over $Y^m$, with structural group $\cy 2$, and the dimension of $Y$ has the form $m = (n+1) -4i$ or $m = (n+1) - 4i+2$, for some $i >0$. By construction, the bundle $ X\to Y$ will have a section so we can embed $Y\subset X$. The surgery problems
$$(f,b)\colon (W, \bd W) \to (X\times I, X\times \bd I),$$
with $\deg f = 1$ and $b\colon \nu_W \to \nu_{X\times I}$ a bundle map covering $f$, will be constructed by glueing a simply-connected Milnor or Kervaire manifold surgery problem \emph{fibrewise} into a tubular neighbourhood of $Y \subset X \times \{1/2\} \subset X\times I$. The details of this construction will be given below.

The basic input is the relation between bordism and homology or $KO$-theory.
\begin{enumerate}
\item (localized at 2) By the work of Thom \cite{thom1} and Conner-Floyd \cite{conner-floyd1}, the Hurewicz map $\Omega^{SO}_m(X,A) \otimes \Ztwo \to H_m(X,A;\Ztwo)$ for oriented bordism is surjective for every  pair $(X,A)$. Similarly, the Hurewicz map
$\cN_m(X,A) \to H_m(X,A; \cy 2)$  for unoriented bordism is surjective, $m\geq 0$.

\item (localized away from 2) There is an isomorphism 
$$h_0\colon \Omega_{k +4*}^{SO}(X)\otimes_{\Omega_*^{SO}(\pt)} \Zodd\xrightarrow{\approx}
KO_k(X; \Zodd)$$
induced by the image of the $KO[1/2]$-fundamental class (see \cite[4.15]{madsen-milgram1}). In this tensor product,  the action $\Omega_*^{SO}(\pt) \to \Zodd$ is given by the index homomorphism  if $*=4i$, and zero if $* \neq 4i$.
\end{enumerate}
For finite groups, there is the following foundational result:

\begin{theorem}[{Wall \cite[\S 7]{wall-V}}]\label{Wall-injective} For $\pi$ a finite group, and $w$ an orientation character, the localization map $L_n(\Zpi, w) \to L_n(\Zpi, w)\otimes \Ztwo$ is injective.
\end{theorem}
We now divide the argument into two cases, since it suffices to show that $I_{n+1}(\pi,w)$ and $C_{n+1}(\pi,w)$ are equal after tensoring with $\Ztwo$ and $\Zodd$ separately. 

\begin{theorem}\label{thm:2-local}
Let $\pi$ be a finitely-presented group, $n\geq 5$, and $w\colon \pi \to \cy 2$  an orientation character. Then
\begin{enumerate}
\item $C_n(\pi,w)\otimes \Ztwo = A_n(\pi,w)\otimes \Ztwo$, and
\item $I_{n+1}(\pi,w)\otimes \Ztwo= A_{n+1}(\pi,w)\otimes \Ztwo$.
\end{enumerate}
\end{theorem}
\begin{corollary}
If $\pi$ is a finite group, then $C_n(\pi,w) = A_n(\pi,w)$ and $I_{n+1}(\pi,w)= A_{n+1}(\pi,w)$, for all $n\geq 5$.
\end{corollary}
\begin{proof}
For finite groups, the only elements of infinite order in $A_n(\pi,w)$ come from the trivial group (see \cite[13B.1]{wall-book}). Hence Theorem A for finite groups follows from the $2$-local version and Theorem \ref{Wall-injective}.
\end{proof}
\begin{proof}[The proof of Theorem \ref{thm:2-local}] We first show that $I_{n+1}(\pi,w)\otimes \Ztwo= A_{n+1}(\pi,w)\otimes \Ztwo$, and note that $(ii) \Rightarrow (i)$ for $n+1\geq 6$.  Alternately, a direct proof  that $C_n(\pi,w)\otimes \Ztwo = A_n(\pi,w)\otimes \Ztwo$, for all $n\geq 5$, can be given along the same lines. The details are similar, but easier, and will be left to the reader. 

 We proceed as outlined above to construct enough inertial elements to generate the domain 
 \eqncount
 \begin{equation}\label{decomposition}
 H_{n+1}(\Kpi^w; \bbL_\bullet)\otimes \Ztwo \xrightarrow{\approx}
\bigoplus_{i > 0} H_{n+1-4i}(\pi; \Ztwo^w) \oplus H_{n+1-4i + 2}(\pi; \cy 2)
\end{equation}
 of the assembly map restricted to $\bbL_\bullet$.
 
Suppose that $\alpha \in H_m(\pi; \Ztwo^w)$ is a given homology class (with twisted coefficients given by the orientation character $w$). Let $\eta$ denote the line bundle over $K(\pi, 1)$ with $w_1(\eta) = w$. By the Thom isomorphism, 
$$\Phi\colon H_m(\pi; \Ztwo^w) \cong H_{m+1}(E, \bd E; \Ztwo)$$
where $E =E(\eta)$ denotes the total space of the disk bundle of $\eta$. Let $h \colon (V^{m+1}, \bd V) \to (E, \bd E)$ be an oriented $(m+1)$-manifold, with reference map to $(E, \bd E)$, whose fundamental class $h_*[V, \bd V] = \Phi(\alpha)$. Now let 
$g\colon Y^m \to \Kpi$ be the transverse  pre-image of the zero section in $E(\eta)$, with $w_1(Y) = g^*(w)$. By construction, $g_*[Y] = \alpha$.

The model surgery problems with target $X \times I$ will be constructed from products
$X = Y^m \times S^{n-m}$. A small tubular neighbourhood of $Y \subset X \times \{1/2\} \subset X \times I$ is homeomorphic to the product $U = Y \times D^{n+1-m}$. By our choice of dimensions, $n+1-m = 4i$,
 for some $i >0$. Let $\varphi\colon (M^{4i}, \bd M) \to (D^{4i}, \bd D^{4i})$
  denote the simply-connected Milnor 
  manifold surgery problems, whose surgery obstructions represent generators of $L_{4i}(\bZ)$.
   Recall that these are smooth surgery problems, with boundary manifolds $\bd M^{4i}$
    smooth homotopy spheres (at least if $4i > 4$), but homeomorphic to the standard sphere by the solution of the Poincar\'e conjecture \cite{smale1}. In dimension $4$, we need the $E_8$-manifold constructed by Freedman \cite{freedman1}.

Now we define $W$ by removing the interior of $U$ from  $X\times I$, and glueing in $Y$ product with  
the Milnor manifold $M^{4i}$.
 The degree 1 map $f\colon W \to X \times I$ is the identity outside of $U = Y \times D^{n+1-m}$, and inside $U$ is given by $\id_Y \times \varphi$.
  Similarly, the bundle map $b\colon \nu_W \to \nu_{X\times I}$ is the identity over the complement of $U$ and given by the simply-connected problem over $U$. We now have a degree 1 normal map $(f,b)\colon W \to X\times I  $ which is the identity on the boundary, hence defines an element  in $ [\Sigma(X), G/TOP]$.

It follows from the characteristic class formula that the surgery obstruction
$$\sigma(f,b)_{(2)} = A_m(\alpha) + {\rm \ lower\  terms}$$
where $A_m = \cI_m$ or $A_m = \kappa_m$, and the ``lower terms" are the images under $A_j$ for $j <m$ (in this formula we have identified $L_{n-4i}= L_n$ by periodicity).

The surgery problems constructed so far are enough to deal with degree $m=(n+1-4i)$ contributions to  $H_{n+1}(\Kpi^w;\bbL_\bullet)$ from the first of the summands in formula (\ref{decomposition}). To realize the $\cy 2$-homology classes $\beta \in H_m(\pi; \cy 2)$ arising from the second summand, we start with a possibly non-orientable manifold $g\colon Y^m \to \Kpi$ with  $g_*[Y] = \beta$. In this case, $n+1 -m = 4i-2$, for some $i >0$.

Sometimes $Y$  can be chosen so that $w_1(Y) = g^*(w)$, and then we  proceed again as above.
The model surgery problems $W$ with target $X \times I$ are constructed from products
$X = Y^m \times S^{n-m}$, by removing a small tubular neighbourhood $U = Y \times D^{n+1-m}$ of $Y \subset X \times \{1/2\} \subset X \times I$. This time we  glue in the product 
$$\id_Y \times \psi\colon Y \times(K^{4i-2}, \bd K) \to Y \times (D^{4i-2}, \bd D^{4i-2})$$
of $Y$ 
with the simply-connected Kervaire surgery problem, whose surgery obstruction represents the  generator of  $L_{4i-2}(\bZ)$. The boundary $\bd K^{4i-2}$
is a smooth homotopy sphere, which is again homeomorphic to the standard sphere, so we may extend by the identity on the complement of $U$.

However, in general we may not have $w_1(Y) = g^*(w)$, and $X$ will be the total space of a certain non-trivial $S^{n-m}$-bundle over $Y$ with structural group $\cy 2$, which we now construct. 

Let $\zeta$ denote the line bundle over $Y$ with $w_1(\zeta) = w_1(Y) + g^*(w)$. Let $\xi = (2i-1)\zeta \oplus (2i-1)\varepsilon$ be the Whitney sum of $(2i-1)$ copies of $\zeta$, together with $(2i-1)$ copies of the trivial line bundle $\varepsilon$. Now let $p\colon X \to Y$ denote the total space of the associated sphere bundle $X = S(\xi)$, and observe that the class $w_1(X) = p^*(g^*(w))$. The fibre sphere has dimension $4i-3 = n-m$. Notice that the bundle $\xi$ has structural group $\cy 2$, and the transition functions defining this bundle operate through the involution denoted
$S^{4i-3}(2i-1)$, meaning the restriction to the unit sphere of the representation
$\bbR^{2i-1}_{+} \oplus \bbR^{2i-1}_{-}$ in which a generator of $\cy 2$ acts as $+1$ on the first subspace and as $-1$ on the second. Since $i >0$ this bundle has non-zero sections, so we may choose an embedding of $Y \subset X$.

We will now show that the Kervaire
sphere $\bd K^{4i-2}$ admits an orientation-reversing involution which is $\cy 2$-equivariantly homeomorphic to $S^{4i-3}(2i-1)$. Recall that $W^{4i-3}(d)$ denotes the Brieskorn variety given by intersecting the solution set of the equation
$$z_0^d + z_1^2 +\dots + z_{2i-1}^2 = 0$$
with the unit sphere in $\bbC^{2i}$, with $d \geq 0$ an odd integer. There is an involution $T_{d}$ on $W^{4i-3}(d)$ given by complex conjugation $z_j \mapsto \bar z_j$ in each coordinate. It is known that $W^{4i-3}(d)$ is a homotopy sphere if $d$ is odd, which is diffeomorphic to the standard sphere  if $d\equiv \pm 1 \pmod 8$, and to the Kervaire sphere $\bd K^{4i-2}$ if $d\equiv \pm 3 \pmod 8$ (see \cite{hirzebruch-mayer1}). 
 The complex conjugation involution extends to the perturbed zero set, which is diffeomorphic to the Kervaire manifold $K^{4i-2}$ if 
$d\equiv \pm 3  \pmod 8$.

\begin{lemma}\label{brieskorn}
The involution $(W^{4i-3}(d), T_d)$, with $d$ odd,  is $\cy 2$-equivariantly homeomorphic to $S^{4i-3}(2i-1)$.
\end{lemma}
\begin{proof}
These involutions were studied by Kitada \cite{kitada1}, who gave necessary and sufficient conditions for 
$(W^{4i-3}(d), T_d)$ to be $\cy 2$-equivariantly diffeomorphic to $(W^{4i-3}(d'), T_{d'})$. We need only the easy part of his argument, namely that $(W^{4i-3}(d), T_d)$ is $\cy 2$-equivariantly normally cobordant to  $S^{4i-3}(2i-1)$ by a normal cobordism which is the identity on a neighbourhood of the fixed set. The remaining surgery obstruction to obtaining an equivariant $s$-cobordism lies in the action of $L_{4i-2}(\bZ[\cy 2],w)\cong \cy 2$ on the relative structure set of the complement of the fixed set. In the smooth category, this action is difficult to determine, but in the topological category the action is trivial (since this element is in the image of the assembly map).
\end{proof}

We can now glue in the simply-connected Kervaire manifold surgery problem in a tubular neighbourhood $U$ of $Y \subset X\times I$. The boundary $\bd U = \widetilde Y \times _{\cy 2} S^{4i-3}$, where $\widetilde Y$ is the double covering of $Y$ given by $w_1(\zeta)$ and the fibre sphere has the action  $S^{4i-3}(2i-1)$. We have a homeomorphism 
$$\widetilde Y \times _{\cy 2} S^{4i-3} \approx \widetilde Y \times _{\cy 2} \bd K^{4i-2}$$
given by Lemma \ref{brieskorn}, and this is used to glue in  $\widetilde Y \times _{\cy 2} K^{4i-2}$ defined by the extension of the complex conjugation involution  over $K^{4i-2}$.
The characteristic class formula shows as before that
the surgery obstruction
$$\sigma(f,b)_{(2)} = \kappa_m(\beta) + {\rm \ lower\  terms}$$
where  the ``lower terms" are the images under $A_j$ for $j <m$.
\end{proof}

\section{The proof of Theorem A (at odd primes) and Theorem B}\label{sec:four}
By Theorem \ref{thm:2-local},  the inertial subgroup and the closed manifold subgroup are both equal to the image of the assembly map,   after localization at $2$. We now localize away from $2$, and this is where we will need to stabilize to identify the image of the assembly map. As above, let $\SA_q(\pi,w)$, $0\leq q \leq 3$, denote the  periodic image of the assembly map, generated by all the $A_n(\pi,w)$ for $n \equiv q \pmod 4$. We will prove:

\begin{theorem}\label{thm:periodic} Let $\pi$ be a finitely-presented group, and $w$ an orientation character. Then 
\begin{enumerate}
\item $I_{n+1}(\pi,w)\otimes \Zodd = C_{n+1}(\pi,w) \otimes \Zodd$, and
\item $\SI_q(\pi,w) = \SC_q(\pi,w)=\SA_q(\pi,w)$.
\end{enumerate} 
\end{theorem}
The procedure in this setting will be similar. The domain of the assembly map is now 
 $$H_{n+1}(\Kpi^w; \bbL_0)\otimes \Zodd \cong \ko_{n+1}(\Kpi^w; \Zodd),$$ and 
 we will write elements arising from $C_{n+1}(\pi,w)$ as a sum with $\Zodd$-coefficients of inertial surgery problems with surgery obstructions in $I_{n+1}(\pi,w)$.

   \begin{remark}\label{rem: connective cover}  
   Note for use in part (ii) that the obstructions of these surgery problems actually come from
   $$ H_{n+1}(\Kpi; \bbL_\bullet)\otimes \Zodd \cong \ko\co_{n+1}(\Kpi; \Zodd),$$
   where $\ko\co_*$ denotes the homology theory given the  $0$-connective cover $\bo\co$ of the spectrum $\bo$. The cofibration of spectra $\bo\co \to \bo \to H\bZ$ induces a long exact sequence
$$ \ldots \to H_{*+1}(X;\Zodd)\to \ko\co_*(X;\Zodd)\to \ko_*(X;\Zodd) \xrightarrow{h} H_*(X;\Zodd) \to \dots$$
where $h\colon  \ko_*(X;\Zodd) \xrightarrow{h} H_*(X;\Zodd)$ is the Hurewicz homomorphism. It follows that, after periodic stabilization,  the same image $\SA_q(\pi,w)$ is generated  from the domains 
   $ H_{n+1}(\Kpi; \bbL_0)\otimes \Zodd \cong \ko_{n+1}(\Kpi; \Zodd)$, with $n+1 \equiv q \pmod 4$. 
\end{remark}
 We need more information about the class $\Delta \in KO^0(\bbL_\bullet\,; \Zodd)$ used in the characteristic class formula for the surgery obstruction (see Section \ref{sec:two}). Recall that there is a Conner-Floyd isomorphism 
 \cite[p.~39]{conner-floyd2} for cobordism
 $$h^0\colon \Omega^{4\ast}(X)\otimes_{\Omega^\ast(\pt)} \Zodd \xrightarrow{\approx} KO^0(X;\Zodd)$$
 which gives the $KO$-theory for a finite complex $X$ in terms of  oriented cobordism away from 2 (this formula uses the identification $MSp[1/2] \simeq MSO[1/2]$).
 \begin{lemma}\label{lem:positive} The class $\Delta \in KO^0(G/TOP; \Zodd)$ is represented by  a formal sum of classes $\widehat\Delta_k\in \Omega^{4k}(G/TOP)\otimes \Zodd$ of positive degrees  $k>0$.
 \end{lemma}
 \begin{proof}
 We first recall the description of $\Delta$ given in \cite[Chap.~4]{madsen-milgram1}.
 For each $k>0$, let $S_k\colon \Omega_{4k}(G/TOP) \to \bZ$ be the homomorphism which assigns to an element $f\colon X \to G/TOP$, the signature difference
 $(\sign M - \sign X)/8$  for the associated surgery problem $M \to X$.
 These are $\Omega_*(\pt)$-module homomorphisms, where $\Omega_*(\pt)$ acts on $\Zodd$ via the signature in dimensions $\equiv 0 \pmod 4$, and zero otherwise.
 
 By \cite[Lemma 4.26]{madsen-milgram1}, the collection $\{S_k\}$ induces a homomorphism
 $$\sigma_0\colon KO_0(G/TOP;\Zodd) \to \Zodd.$$
  The proof uses the Conner-Floyd isomorphism and an inverse limit argument over finite skeleta of $G/TOP$. Now one applies the universal coefficient formula for $KO$-theory \cite[(2.8)]{yosimura1}, and in particular the isomorphism
 $$\eval\colon KO^0(G/TOP;\Zodd) \to \Hom_{\bZ}(KO_0(G/TOP;\Zodd), \Zodd)$$
 to get the element 
 $$\Delta \in \wKO^0(G/TOP;\Zodd) = [G/TOP, BO[1/2]],$$
  with $\eval(\Delta) = \sigma_0$ (see \cite[p.~86]{madsen-milgram1} and the proof of \cite[(4.26)]{madsen-milgram1} for the assertion that $\eval$ is an isomorphism). Note that the element $\Delta$ lies in reduced $KO^0$  since the homomorphisms $S_k$ have positive degree.  The associated map
 $\sigma\colon G/TOP[1/2] \to BO[1/2]$ is the Sullivan homotopy equivalence (see \cite[4.28]{madsen-milgram1}). 
 
By the Conner-Floyd isomorphism for cobordism, there is a unique element
 $$\widehat\Delta  \in \wO^{4\ast}(G/TOP)\otimes_{\Omega^\ast(\pt)} \Zodd,$$
 such that
 $h^0(\widehat\Delta) = \Delta \in \wKO^0(G/TOP;\Zodd)$. 
We may consider this tensor product as a quotient of the corresponding direct product, and represent elements as infinite formal sums.
  
Quillen \cite[Theorem~5.1]{quillen_1971} proved that the reduced (complex) cobordism group of a connected finite complex is generated as a $U^\ast(\pt)$-module by elements in strictly positive dimensions, and the same is true for oriented bordism as an $\Omega^\ast(pt)$-module after inverting 2. Therefore, $\widehat \Delta$ is represented in the tensor product by a formal sum of elements 
 $$\Delta_k  \in \wO^{4k}(G/TOP;\Zodd)$$
 with $k>0$. 
 \end{proof}

\begin{proof}[The proof of Theorem \textup{\ref{thm:periodic}}, part \textup{(i)}]
We  consider an element 
$$\alpha
 =(f^*(\Delta) \cap  [V]_{\bo})
  \in \ko_{*}(V^w; \Zodd),$$ given by  $f\colon V \to \bbL_{\bullet}$ and reference map $c\colon V \to \Kpi$, with $c^*(w) = w_1(V)$,
  whose image under the assembly map gives an element of $C_{n+1}(\pi,w)\otimes \Zodd$. By Lemma \ref{lem:positive} and Poincar\'e duality for bordism theory \cite{atiyah_1961},  we can express 
 $$\alpha = \sum _{k> 0} a_k [Y^{n+1-4k}, g_{k}]$$ as a finite $\Zodd$-linear combination of manifolds $g_{k}\colon Y^{n+1-4k} \to V$, with $g_k^*(w_1(V)) = w_1(Y)$, and coefficients $a_k \in \Zodd$ for  $k>0$. 
 Let $g\colon Y^{n+1-4k} \to  \Kpi$ be a manifold with reference map (by composing with  $V \to \Kpi$), such that 
 $k$ is the smallest integer with $a_k \neq 0$. Hence $a_k\cdot g_*([Y]_{\bo}) = \alpha + \ {\rm lower \ terms}$.  
 
We will now construct an element in $I_{n+1}(\pi,w)$. We write $n+1 - m =4k$, and define $X = Y^m \times S^{4k-1}$. The surgery problem 
$$(h,b)\colon (W^{n+1}, \bd W) \to (X \times I, X\times \bd I)$$
 will be constructed as before, by gluing in the Milnor manifold surgery problem $$(M^{4k}, \bd M) \to (D^{4k}, \bd D^{4k})$$ fibrewise along a tubular neighbourhood $U \subset X \times I$ of $Y \subset X \times \{1/2\}$ in the interior of $X \times I$.
 Let $h\colon \Sigma(X) \to \bbL_\bullet$  also denote the normal invariant of $(h,b)$, which factors as the composite
 $$h\colon \Sigma(X) \xrightarrow{project} Y \times D^{4k}/Y \times S^{4k-1} \xrightarrow{1\times \varphi} \bbL_\bullet,
$$
where $\varphi\colon S^{4k} \to \bbL_\bullet$ is the normal invariant of the Milnor problem (i.e.~ the generator of $\pi_{4k}(\bbL_\bullet) = \bZ$).
  The
 characteristic class formula
$$a_k\cdot \sigma(h)_{(odd)} = a_k \cdot A_*(g\times 1)_*\left (h^*(\Delta) \cap [Y\times S^{4k}]_{\bo} \right) = A_*(\alpha) + \ {\rm lower \ terms},$$
since
$h^*(\Delta) \cap [Y\times S^{4k}]_{\bo} = [Y]_{\bo}$, and $a_k \cdot g_*([Y]_{\bo}) = \alpha+ \ {\rm lower \ terms}$.
This completes the proof of part(i) of Theorem \ref{thm:periodic}.
\end{proof}
\begin{remark}
This formula is consistent with the rationalization of the calculation at 2.
Note that the Poincar\'e dual $\cL(Y)$ of the $\cL$-genus gives the rational part of the $\bbL^0$-theory fundamental class $[Y]_{\bbQ} \cap \cL(Y) \in H_{m-4*}(Y;\bbQ)$, by 
\cite[25.17]{ranicki-assembly}. 
Under the equivalence $ \bbL^0\otimes \Zodd \simeq \bo(\Zodd)$, the fundamental class $[Y]_{\bbL^0} \in H_m(Y; \bbL^0\otimes \Zodd)$ maps to $[Y]_{\bo} \in \ko_m(Y;\Zodd)$. 
\end{remark}
\begin{proof}[The proof of Theorem B]
By Theorem \ref{thm:2-local} it is enough to show that $\SI_q(\pi,w)\otimes \Zodd = \SA_q(\pi,w)\otimes \Zodd$, for $0 \leq q \leq 3$.
We represent an arbitrary element of $\SA_q(\pi, w)$ by the image $A_m(\alpha)\in A_{m}(\pi,w)\otimes \Zodd$ under the assembly map of an element $\alpha \in \ko_{m}(\Kpi^w; \Zodd)$,   where $m\equiv q\pmod 4$ and $m\geq 5$.

Since $KO$-homology satisfies the wedge axiom, the group $\ko_{m}(\Kpi^w; \Zodd)$ is the direct limit of the $\ko$-homology of the finite skeleta of the classifying space $\Kpi$. By periodic stabilization if necessary, we may assume that 
$$\alpha \in \Image\big (\ko_m(X^w;\Zodd)\to \ko_m(\Kpi^w;\Zodd)\big ),$$
 where $X$ is a suitable  finite skeleton of $\Kpi$ with $\dim X < m$. But then,
$\ko_m(X^w;\Zodd) = KO_m(X^w;\Zodd)$.
 By using the Conner-Floyd theorem \cite[4.15]{madsen-milgram1} and the $4$-fold periodicity $KO_{m+4k}\cong KO_m$, for $k>0$, we can express 
 $$\alpha = \sum _{k>0} a_k [Y^{m+4k}, g_{k}]$$ as a finite $\Zodd$-linear combination of the images of fundamental classes $[Y]_{\bo}$ of manifolds $g_{k}\colon Y^{m+4k} \to \Kpi$, with $g_k^*(w) = w_1(Y)$ and $a_k \in \Zodd$.  The same construction (glueing in  Milnor manifold surgery problem $(M^{4l}, \bd M) \to (D^{4l}, \bd D^{4l})$) used in the proof of part (i), when applied to a typical element $(g_k)_*[Y]_{\bo}\in \ko_{n+1}(\Kpi;\Zodd)$, produces an inertial surgery problem with obstruction in $I_{m + 4k + 4l}(\pi, w)$. It follows
 that the periodic stabilization of the image $A_m(\alpha)$ is the sum of surgery obstructions of elements in various inertial subgroups $I_{n+1}(\pi, w)$, for $n+1 \equiv m \pmod 4$. Therefore $\SA_m(\pi,w)\otimes \Zodd = \SI_{m}(\pi,w)\otimes \Zodd$.
  \end{proof}
\begin{example}\label{ko-example}
We give an example (based on work of Conner-Smith \cite{conner-smith1} and Johnson-Wilson \cite{johnson-wilson1})
to show that,  in a given dimension $n$, the image of the assembly map $A_n(\pi,w)$ is not always equal to the closed manifold subgroup $C_n(\pi,w)$. In particular, this contradicts \cite[18.6(i)]{ranicki-assembly}, and shows that for a suitable finite complex $X$ the
elements of
$H_n(X; \bbL_{\bullet})$ are not always represented by closed manifold surgery problems. 

We will need \cite[Prop.~2.6]{hhausmann1}, which is
a variation of the Kan-Thurston theorem  \cite{kan-thurston1},  \cite{baumslag-dyer-heller1}.  For any finite complex $X$, there exists a finitely-presented group $\Gamma_X$ with $B\Gamma_X$ of dimension $\leq \dim X$, and an epimorphism $\varphi\colon \Gamma_X \to \pi_1(X)$ with perfect kernel. Moreover, there is a lifting
$\widetilde \alpha_X\colon X \to (B\Gamma_X)^{+}_{\ker\varphi}$ of the classifying map $\alpha_X\colon X \to \Kpi_1(X)$ which is a homotopy equivalence. In other words, $X$ is obtained by applying the Quillen plus construction to $B\Gamma_X$. It follows (from the Atiyah-Hirzebruch spectral sequence) that $\ko_*(X) \cong \ko_*(B\Gamma_X)$. 
Since the image of $\ko\co_*(X)\to \ko_*(X)$ equals the kernel of the Hurewicz homomorphism, it is therefore enough to produce the following example:
\begin{lemma}\label{ex: finite complex}
There exists a finite complex $X$ such that the natural map
$$\Omega^{SO}_m(X)\otimes \Zodd \to \ko_m(X;\Zodd)$$
is not surjective onto the kernel of the Hurewicz homomorphism
$$\ko_m(X;\Zodd) \to H_m(X;\Zodd),$$
 in some dimension $m\geq 5$.
\end{lemma}
\begin{remark} Recall that $H_n(X; \bbL_{\bullet})\otimes \Zodd \cong \ko\co_n(X;\Zodd)$.
This example also shows that the elements of $H_n(K; \bbL_{\bullet})$ are not always represented by closed topological manifold surgery problems. The reason is that the natural maps
$$\Omega^{SO}_m(X)\otimes \Zodd \to \ko_m(X;\Zodd)$$
and
$$\Omega^{STOP}_m(X)\otimes \Zodd \to \ko_m(X;\Zodd)$$
have the same image, for $X$ a finite complex. 
This is essentially a result of Hodgkin-Snaith (see \cite[5.22, 5.24]{madsen-milgram1}:  the natural maps
$$\widetilde{KO}^*(MSTOP[1/2]) \xrightarrow{\approx} \widetilde{KO}^*(MSPL[1/2])
\xrightarrow{\approx} \widetilde{KO}^*(MSO[1/2])$$
are all isomorphisms (using the fact that $MSTOP$ and $MSPL$ are the same away from 2). The map
$$\Omega^{SPL}_m(X)\otimes \Zodd \to \ko_m(X;\Zodd)$$
is given by the map on homotopy groups induced by the map
$$MSPL[1/2]\, \wedge X \to  BO[1/2]\, \wedge X$$
provided by smashing over the identity on $X$ with  the 
Sullivan orientation $\Delta_{SPL}$ (see Madsen-Milgram \cite[p.~100]{madsen-milgram1},  formula 5.2 and Lemma 5.3 to compare with $\Delta_{SO}$).
\end{remark}

We first discuss the analogous question for complex bordism.
In a series of papers Conner and Smith studied the natural map
$$h\colon \Omega^U_m(X) \to \ku_m(X)$$
from complex bordism to connective complex $K$-homology theory induced by the $K$-theory orientation (see Conner and Floyd \cite{conner-floyd2}).  The coefficient ring $\Omega^U_*(\pt) $  is a polynomial ring on even dimensional generators and 
the coefficient ring $\ku_*(\pt) = \bZ[t]$, where $\deg t = 2$. Multiplication by $[\CP^1]$ on $\Omega^U_*(X)$ corresponds under $h$ to multiplication by $t$ on  $\ku_m(X)$ (see Stong \cite{stong1}).

 Conner and Smith show in \cite[Theorem~10.8]{conner-smith1} that, for any finite complex $X$ and for any class $a\in \ku_m(X)$, there exists an integer $n = n(a) \geq 0$ such that $t^n a \in \Image(\Omega^U_{m+2n}(X) \to \ku_{m+2n}(X)$. However, according to a result of Johnson and Smith \cite[Theorem~1]{johnson-smith1}, for a finite complex $X$  the natural map
$\Omega^U_*(X) \to \ku_*(X)$ is onto if and only if the projective dimension of $\Omega_*^U(X)$ over $\Omega^U_*(\pt)$ is $\leq 2$.

 On the other hand, by a result of Conner and Smith \cite[Theorem~5.1]{conner-smith3},  a large $N$-skeleton $X$ of $K(Z/p, n)$, $p$ an odd prime, will have a large homological dimension over $MU$. We may pick one with $\hdim_{\Omega^U_*(\pt)}\Omega^U_*(X) \geq 3$, and with $p$ an odd prime, and both $n$ and $N$ fairly large (see also \cite[p.~854]{smith0} for an explicit example).
Such a finite complex $X$ gives an example to show that the natural map
$$\Omega^{U}_m(X)\otimes \Zodd \to \ku_m(X;\Zodd)$$
is not surjective in some dimension $m\geq 5$ (the dimension $m$ can always be raised by suspension of the example).  However, in this case the  Hurewicz map 
$\ku_m(X;\Zodd) \to H_m(X; \Zodd)$ may be injective, and non-realizable elements from
$\ku\co_m(X;\Zodd)$ may not exist. It does however show that stabilizing is sometimes actually necessary to realize elements of $\ku_m(X;\Zodd)$ by fundamental classes of almost complex manifolds.

\medskip
This theme was definitively addressed by Johnson and Wilson \cite{johnson-wilson1} 
using the $p$-local Brown-Peterson homology theories $BP_*$. Recall that for a given prime $p$, the
localized complex bordism spectrum  $MU_{(p)}$ splits as a wedge of shifted copies of $BP$, and the coefficient ring $BP_*(\pt) = \Zp[v_1, v_2, \dots, v_n, \dots]$, where $v_n$ has degree $2p^n-2$ (see \cite[Theorem 1.3]{brown-peterson1}). There are also associated theories $BP\la n\ra$, for $n \geq 1$, with $BP\la n\ra_*(\pt) = \Zp[v_1, \dots, v_n]$, constructed by Wilson \cite{wilson1}, \cite{wilson2}. The spectrum $BP\la 1\ra$ is a wedge summand of $\bu_{(p)}$ under the identification $v_1 = t^{p-1}$ (see \cite[2.7]{johnson-wilson1}). 

The Brown-Peterson theory also shows that at odd primes the spectrum $MSO_{(p)}$  splits as a wedge of shifted copies of $BP$ (again see \cite[Theorem 1.3]{brown-peterson1}), and that $\bo_{(p)}$ splits off $BP\la 1\ra$ as a wedge summand. Furthermore, 
the natural map 
$$\Omega^{SO}_m(X)\otimes \Zp = \pi_m(MSO_{(p)}\wedge X) \to \pi_m(\bo_{(p)}\wedge X) =\ko_m(X; \Zp)$$ induced by the map of spectra $MSO_{(p)} \to \bo_{(p)}$  has a corresponding  splitting, and therefore contains the map $BP_m(X) \to 
BP\la 1\ra_*(X)$ as a direct summand.
In short, it will be enough to find an example for $BP$-theory. 

\begin{proof}[The proof of Lemma \ref{ex: finite complex}] We fix an odd prime $p$.  Consider the following stable complexes\footnote{I am very much indebted to David Johnson and W.~Stephen Wilson for providing this example.}. There is a self-map $f\colon \Sigma^{2(p-1)}X_0 \to X_0$ realizing multiplication by $v_1$ on $BP_*(X_0)$, where $X_0=M(p)$ denotes the mod $p$ Moore spectrum.The cofibre of $f$ has a finite complex model  (usually called $V(1)$, see Smith \cite[Theorem~1.5]{smith0}).

Let $X_1 = M(p, v_1^n)$ denote the cofibre of the $n$-fold iterate $f^n$, $n \geq 2$, which realizes
multiplication by $v_1^n$. Therefore
$BP_*(X_1) = BP_*/(p, v_1^n)$.  Furthermore, by Hopkins and J.~H.~Smith \cite[Theorem~9]{hopkins-smith1},  there exists a self-map $g\colon \Sigma^{2m(p^2-1)}X_1 \to X_1$ realizing multiplication by $v_2^m$ on $BP_*(X_1)$,  for some large $m\geq 1$. We let
$$X := M(p, v_1^n, v_2^m)$$
 denote the cofibre of $g$. Then $BP_*(X) = BP_*/(p, v_1^n, v_2^m)$. However, multiplication by $v_2$ is zero on $BP\la 1\ra_*$-homology, and so 
$$BP\la 1\ra_*(X) \cong BP\la 1\ra_*(X_1) \oplus BP\la 1\ra_{*- k -1}(X_1),$$
where $k = 2m(p^2-1)$.
Since $BP_*(X)$ is concentrated in even dimensions, the odd dimensional classes are not in the image of the natural map $BP_*(X) \to BP\la 1\ra_*(X)$. 

For the example, we may choose any non-zero multiple of $v_1$ in the odd-dimensional summand of  $BP\la 1\ra_{*}(X)$. Such an element is in the kernel of the Hurewicz homomorphism (see \cite[p.~328]{johnson-wilson1}), and is annihilated by $v_1^{n-1}$, but is not in the image from $BP_*(X)$.
  \end{proof}
\end{example}

\providecommand{\bysame}{\leavevmode\hbox to3em{\hrulefill}\thinspace}
\providecommand{\MR}{\relax\ifhmode\unskip\space\fi MR }
\providecommand{\MRhref}[2]{%
  \href{http://www.ams.org/mathscinet-getitem?mr=#1}{#2}
}
\providecommand{\href}[2]{#2}

\end{document}